\documentclass[reqno]{amsart}
\usepackage[utf8]{inputenc}
\usepackage[utf8]{inputenc}
\newtheorem{theorem}{Theorem}[section]
\newtheorem{proposition}[theorem]{Proposition}
\newtheorem{conjecture}[theorem]{Conjecture}

\newtheorem{lemma}[theorem]{Lemma}
\newtheorem{corollary}[theorem]{Corollary}
\newtheorem{remark}[theorem]{Remark}

\author[I. G. Gargate]{Ivan Gonzales Gargate}
\address{Universidade Tecnol\'ogica Federal do Paran\'a, campus Pato Branco}
\email{ivangargate@utfpr.edu.br}

\author[T. C. de Mello]{Thiago Castilho de Mello}
\address{Universidade Federal de S\~ao Paulo, Instituto de Ci\^encia e Tecnologia}
\email{tcmello@unifesp.br}

\title[A new approach to the Lvov-Kaplansky conjecture through gradings]{A new approach to the Lvov-Kaplansky conjecture through gradings}

\date{}


\begin{document}

\begin{abstract}
	In this paper we consider images of (ordinary) noncommutative polynomials on matrix algebras endowed with a graded structure. We give necessary and sufficient conditions to verify that some multilinear polynomial is a central polynomial, or a trace zero polynomial, and we use this approach to present an equivalent statement to the Lvov-Kaplansky conjecture.
\end{abstract}

\maketitle

\section{Introduction}

Let $K$ be a field, $X$ be a countable set and $K\langle X \rangle$ denote the free associative algebra, freely generated by a set $X$ (i.e., the set of noncommutative polynomials in the variables of $X$). If $f(x_1, \dots, x_m)\in K\langle X\rangle$, and $A$ is a $K$-algebra, $f$ defines a map (also denoted by $f$)
\[\begin{array}{cccc}
f: & A^m & \longrightarrow & A\\
  & (a_1, \dots, a_m) & \longmapsto & f(a_1, \dots, a_m)
\end{array}\] by evaluation of variables on elements of $A$. One may ask what is the image of a given polynomial, or either which subsets of $A$ are the image of some polynomial in $K\langle X \rangle$.

Problems of this type were attributed to Kaplansky, when $A=M_n(K)$. Also, if $f$ is a multilinear polynomial, Lvov asked if the image of $f$ is always a vector subspace of $M_n(K)$ (see \cite[Problem 1.98]{Dniester}).
One can prove that if the answer to this question is true, then the image of $f$ must be one of the following:
\[\{0\}, \quad K, \quad sl_n(K) \quad \textrm{ or } \quad M_n(K).\]
Here $K$ represents the set of scalar matrices and $sl_n(K)$ the set of trace zero matrices. This is now known as the Lvov-Kaplansky conjecture:
\begin{conjecture}[Lvov-Kaplansky conjecture]
	If $f(x_1, \dots, x_m) \in K\langle X \rangle$ is a multilinear polynomial, then its image on $M_n(K)$ is $\{0\}$, $K$, $sl_n(K)$ or $M_n(K)$.
\end{conjecture}
A solution to the above conjecture is known only for $m=2$ or $n=2$ (under some restrictions on the base field $K$), and there are partial results for $n=3$ and $m=3$, see \cite{survey} for a compilation of known results about this conjecture and other topics related to images of polynomials on algebras. This kind of problem was also studied for other algebras, not necessarily associative. For instance, when $A$ is the algebra of upper triangular matrices, $UT_n(K)$, or its subset of strictly upper triangular matrices, a complete solution is known under some conditions on the base field $K$ (see \cite{GargatedeMello, Wang_nxn, Fagundes}). If $A$ is the quaternion algebra, a complete solution was given in \cite{MalevQ}. In the non-associative setting, a complete solution is known for the octonion algebra \cite{MalevO} and for some classes of Jordan Algebras, including the (simple) algebra of a symmetric bilinear form \cite{MalevJ}.

A related conjecture is the so called Mesyan Conjecture (see \cite{Mesyan} and \cite{MesyanRestated}). It is a weaker version of Lvov-Kaplansy conjecture and it can be stated as follows

\begin{conjecture}[Mesyan conjecture]
    Let $K$ be a field $n\geq 2$ and $m\geq 1$ be integers, and let $f(x_1, \dots, x_m)$ be a nonzero multilinear polynomial in $K\langle x_1, \dots, x_m\rangle$. If $m\leq 2n-1$ then the image of $f$ on $M_n(K)$ contains $sl_n(K)$.
\end{conjecture}

Such conjecture was proved for $m\leq 4$ (\cite{Mesyan, BuzinskiWinstanley, MesyanRestated}).

The theory of image of polynomials on algebras is strongly connected with the theory of algebras with polynomial identities (PI-algebras). For instance a polynomial $f$ is a polynomial identity (PI) of $A$ if  its image is $\{0\}$, and $f$ is a central polynomial for $A$ if the image of $f$ is contained in the center of $A$.

Also, the theory of polynomial identities provides interesting results on which one can rely to study images of polynomials. For instance, in the solution of the case $n=2$ of Lvov-Kaplansky conjecture, a key argument is based on the fact that the algebra of generic matrices is a domain. Recall that the algebra of generic matrices is an algebra generated by matrices in which the entries are distinct variables (see \cite[Chapter 7]{Drensky}). It is well-known that such algebra is isomorphic to the quotient algebra $\frac{K\langle X \rangle}{Id(M_n(K))}$, where $Id(A)$ denotes the ideal of polynomial identities of an algebra $A$, i.e., the algebra of polynomials module the identities of $A$.

An usual approach in studying polynomial identities of algebras, is the use of gradings, specially after the seminal work of Kemer \cite{Kemer}, where gradings were used to give a positive solution to the Specht problem in characteristic zero. Gradings provide an interesting approach to study identities, once one usually reduces the problem of evaluating elements in the whole algebra to evaluating elements in some particular vector subspaces. For instance, one can show that if two algebras satisfy the same graded polynomial identities, then they satisfy the same ordinary identities.

Following this line of research, it is a natural step in studying images of polynomials, to consider images of graded polynomials on graded algebras. This was done recently in the papers \cite{CentronedeMello, PlamenPedro} for full and upper triangular matrices. In this case, one needs to work with the so called \emph{graded polynomials} (see \cite{CentronedeMello}).

Although in the present paper we  are still considering gradings, our approach here is somewhat different to the above mentioned papers. We will use gradings and images of multilinear polynomials, to obtain results about ordinary polynomial identities and central polynomials, and to present an equivalent statement to the Lvov-Kaplansky conjecture.

The paper is organized as follows: in section 2 we present the preliminary concepts and results needed in the paper. In section 3, we present an statement (Theorem \ref{identities}) that gives necessary and sufficient conditions for a multilinear polynomial $f$ to be an identity for $M_n(K)$. In section 4, we present necessary and sufficient conditions for a multilinear polynomial to be a central polynomial and we present an equivalence to the Lvov-Kaplansky conjecture and in section 5, we give some applications of our results.
\section{Preliminaries}

Let $K$ be a field and $G$ be a group (with multiplicative notation). We say a $K$-algebra $A$ is a $G$-graded algebra, if there exist subspaces $A_g$, for each $g\in G$ such that 
\[A = \oplus_{g\in G}A_g\] 
and  $A_gA_h\subseteq A_{gh}$ for each $g, h\in G$. The elements of the subspace $A_g$ are called homogeneous of degree $g$.

Gradings on a matrix algebras $A=M_n(K)$ have been completely classified if $K$ is a an algebraically closed field of characteristic zero. Essentially, they can be presented as a tensor product of a matrix algebra with a kind of grading called \emph{elementary} and a graded division algebra \cite{BahturinZaicev}. 
A particular kind of elementary grading on $M_n(K)$ is the so called Vasilovsky grading over the group $\mathbb{Z}_n$. In such grading, the component $g$ is the subspace spanned by the matrices $E_{i,j}$ such that $j-i = g $ in $\mathbb{Z}_n$. Here $E_{i,j}$ denotes the matrix with 1 in entry $(i,j)$ and 0 elsewhere (if $i$ or $j\not \in \{1, \dots, n\}$, we consider their representatives module $n$). 
So for the Vasilovsky grading, the component $t$ is as below

\[\left(
    \begin{array}{ccccccc}
      0           & \cdots & 0       & a_{1,t+1} & 0         & \cdots & 0            \\
      0           & \cdots & 0       & 0       & a_{2,t+2} & \cdots & 0            \\
      \vdots      & \ddots & \vdots  & \vdots  & \vdots    & \ddots & \vdots       \\
      0           & \cdots & 0       & 0       & 0         & \cdots & a_{n-t,n}    \\
      a_{n-t+1,1} & \cdots & 0       & 0       & 0         & \cdots & 0            \\
      0           & \ddots & 0       & 0       & 0         & \cdots & 0            \\
      0           & \cdots & a_{n,t} & 0       & 0         & \cdots & 0            \\
    \end{array}
  \right),
\]

The graded polynomial identities of matrices with this grading have been described in \cite{Vasilovsky} for fields of zero characteristic and in \cite{Azevedo} for infinite fields. Also, the central polynomials  were described in \cite{Brandao}.

When dealing with polynomial identities over fields of characteristic zero, the multilinear polynomials play an important role. Namely, the ideal of identities of a given algebra $A$ is generated (as a T-ideal) by its multilinear polynomial identities. In particular, the sequence of codimensions of a given PI-algebra provides an important way to study (asymptotically) identities of a given algebra or variety of algebras.

A polynomial $f(x_1, \dots, x_m)\in K\langle X \rangle$ is called multilinear if it can be written as
\[f(x_1, \dots, x_m) = \sum_{\sigma\in S_m}\alpha_\sigma x_{\sigma(1)} \cdots x_{\sigma(m)},\]
for some $\alpha_{\sigma}\in K$. Here $S_m$ stands for the symmetric group on $\{1, \dots, m\}$

Since the Lvov-Kaplansky conjecture asks if the image of a multilinear polynomial is a vector subspace, it is important to know what kind of structure such a set has.

One can easily see that the image of a polynomial $f$ on an algebra $A$ is invariant under automorphisms of such algebra. Indeed, one just need to notice that if $\varphi$ is an automorphism of $A$, for any $a_1, \dots, a_m\in A$ one has 
\[\varphi(f(a_1, \dots, a_m)) = f(\varphi(a_1), \dots, \varphi(a_m)).\]
In the case of $A=M_n(K)$ this is the same as saying that the image of a polynomial is invariant under conjugation.

Also, if the polynomial $f$ is linear in one of its variables, then the image of $f$ is closed under scalar multiplication.

The notion of \emph{invariant cone} was defined in \cite{K-BMR} for matrix algebras. We say that a subset of $A$ is an \emph{invariant cone} of $A$ if it is closed under conjugation and scalar multiplication.

One such cone is said \emph{irreducible} if it contains no proper invariant cone.

Observe that if $S$ is an irreducible invariant cone in $A$, then if the image of $f$ intersects $S$ nontrivially, then $S$ is contained in the image of $f$.

\section{A new approach to polynomial identities of matrices}

In this section we present a new approach to study polynomial identities on matrices, that relies on the fact that the image of a polynomial is invariant under endomorphisms of algebras.

In order to present our main result, we consider $A=\oplus_{g\in \mathbb{Z}_n}A_g$ to be the algebra of $n\times n$ matrices over $K$ endowed with the Vasilovsky grading and we consider the following statement for a multilinear polynomial $f\in K\langle X \rangle$.

\begin{enumerate}
    \item [(S0)] If $a_1, \dots, a_m\in M_n(K)$ are homogeneous matrices satisfying $\sum _{i=1}^m\deg(a_i)=0$ then $f(a_1,\dots, a_m)=0$.
\end{enumerate}

\begin{theorem}\label{identities}
    Let  $f(x_1,\dots, x_m)  \in K\langle X \rangle$ be a multilinear polynomial. Then $f$ is a polynomial identity for $M_n(K)$ if and only if $f$ satisfies \emph{(S0)}.
\end{theorem}

\begin{proof}

The ``only if" part is trivial. 

Let us assume $f$ satisfies (S0), that is, for any homogeneous $a_1,\cdots,a_m$ satisfying $\sum_{i=1}^m \deg(a_i)= 0$ we have $f(a_1, \dots, a_m)=0$.

Taking arbitrary elements $b_1,\dots,b_m\in M_n(K)$ and writing them as \[b_j=\displaystyle \sum_{i\in \mathbb{Z}_n}a_i^{(j)}, \textrm{ for each } j\in \{1, \dots, m\},\]
with $\deg(a_i^{(j)})=i$, for $j\in \{1, \dots, m\}$ and $i\in \mathbb{Z}_n$, the multilinearity of $f$ implies  that we may open the brackets to obtain
    \begin{align*}
        f(b_1,\dots,b_m) & =\sum_{i_j\in \mathbb{Z}_n}f(a^{(1)}_{i_1}, \dots, a^{(m)}_{i_m})\\
                         & = \sum_{i_1+\cdots +i_m=\overline{0}}f(a^{(1)}_{i_1}, \dots, a^{(m)}_{i_m}) + \sum_{i_1+\cdots +i_m\neq \overline{0}}f(a^{(1)}_{i_1}, \dots, a^{(m)}_{i_m})
    \end{align*}

Then we obtain \[f(b_1,\dots,b_m) = \sum_{i_1+\cdots +i_m\neq \overline{0}}f(a^{(1)}_{i_1}, \dots, a^{(m)}_{i_m})\]
The above means that the image of $f$ on $M_n(K)$ is a subset of the set of zero diagonal matrices (also known as \emph{hollow matrices}). But Theorem 2 of \cite{Fillmore} (see also \cite{Fillmorerestated}) implies that if $a\in M_n(K)$ ($n\geq 2$) is nonzero and has zero diagonal, then it is conjugated to a matrix whose $(1,1)$ entry is nonzero. 
Since the image of $f$ is invariant under conjugation, it follows that $f(b_1,\dots,b_m)=0$, and $p$ is a polynomial identity for $M_n(K)$.
\end{proof}

Notice that the above theorem provides a weaker condition to verify whether a multilinear polynomial is an identity for $M_n(K)$. If $f$ is a multilinear polynomial, in order to verify it is an identity for $M_n(K)$, it is enough to verify it vanishes under the evaluation of $f$ in a set which generates $M_n(K)$. As a consequence, we have

\begin{corollary}
    Let $f(x_1, \dots, x_m)\in K\langle X \rangle$ be a multilinear polynomial. Then $f$ is a polynomial identity for $M_n(K)$ if and only if for any set of homogeneous elements $a_1, \dots, a_k$ which spans $M_n(K)$, $f(a_{i_1}, \dots, a_{i_m}) = 0$ whenever $i_j\in \{1, \dots, k\}$ and  $\sum_{j=1}^m\deg(a_{i_j})=0$.
\end{corollary}

The above theorem and more specifically its corollary may be a useful tool when considering computational approaches to polynomial identities (for instance, as in \cite{Bondari}), since it reduces the computational effort to verify if some multilinear polynomial is an identity.

\begin{remark}
    An analogous result (with completely similar proof) holds when considering $M_n(K)$ endowed with any elementary $G$-grading by an abelian group $G$, whose neutral component is the set of diagonal matrices. 
\end{remark}

\section{An equivalence to the Lvov-Kaplansky Conjecture}

Let $f(x_1,\dots,x_m)\in K\langle X \rangle$ be a multilinear polynomial and let $A=\oplus_{g\in \mathbb{Z}_n}A_g$ be  the algebra of $n\times n$ matrices over $K$ endowed with the Vasilovsky grading. Let us consider the following statements concerning the image of $f$ on $A$:

\begin{enumerate}
    \item [(S1)] If $a_1, \dots, a_m\in M_n(K)$ are homogeneous matrices satisfying $\sum_{i=1}^m \deg(a_i)\neq 0$ then $f(a_1,\dots, a_m)=0$.
    \item [(S2)] If $a_1, \dots, a_m\in M_n(K)$ are homogeneous matrices satisfying $\sum _{i=1}^m\deg(a_i)=0$ then $tr(f(a_1,\dots, a_m))=0$.
\end{enumerate}

In this section we show that statements as the above can be used to have a better understanding of the image of $f$. In particular, the above are useful to characterize central polynomials and polynomial identities for matrix algebras.

We recall that we denote by $K$ the set of scalar matrices. In particular, $Im(f)\subseteq K$ means that $f$ is a central polynomial for $M_n(K)$.

\begin{lemma}\label{S1S2}
Let $f(x_1,\dots,x_m)\in K\langle X \rangle$ be a multilinear polynomial. Then 
    \begin{enumerate}
        \item $Im(f)\subseteq K$ if and only if $f$ satisfies \emph{(S1)}.
        \item $Im(f)\subseteq sl_n(K)$ if and only if $f$ satisfies  \emph{(S2)}.
    \end{enumerate}
\end{lemma}

\begin{proof}

\begin{enumerate}
    \item 
The proof is similar to the proof of Theorem \ref{identities}. Again, the ``only if" part is trivial.

Assume $f$ satisfies {(S1)}, that is,  for any homogeneous $a_1,\cdots,a_m$ satisfying $\sum_{i=1}^m \deg(a_i)\neq 0$ we have $f(a_1,\cdots,a_m)=0$. Let $b_1,\cdots, b_m\in M_n(K)$ and write them as 
\[b_j=\displaystyle \sum_{i\in \mathbb{Z}_n}a_i^{(j)},\]
with $\deg(a_i^{(j)})=i$, for $j\in \{1, \dots, m\}$ and $i\in \mathbb{Z}_n$.
Since $f$ is multilinear
\begin{align*}
    f(b_1,\cdots,b_m)&=\displaystyle \sum_{i_1+\cdots+i_m=0}f(a_{i_1}^{(1)},\cdots,a_{i_m}^{(m)}) +\displaystyle \sum_{i_1+\cdots+i_m\neq 0}f(a_{i_1}^{(1)},\cdots,a_{i_m}^{(m)})\\
    &=\displaystyle \sum_{i_1+\cdots+i_m=0}f(a_{i_1}^{(1)},\cdots,a_{i_m}^{(m)}).
\end{align*}

This would imply that $Im(f)$ lies in the subset of diagonal matrices. But it is well know that if a diagonal matrix is not scalar, then it is conjugated to a nondiagonal matrix, which cannot occur, since $Im(f)$ is invariant under conjugation. As a consequence, $Im(f)\subseteq K$, i.e., $f$ is a central polynomial.

\item  Again, the only if part is trivial.

Let us now assume that $f$ satisfies {(S2)}, that is, for any homogeneous $a_1,\cdots,a_m$ satisfying $\sum_{i=1}^m \deg(a_i)= 0$ we have $tr(f(a_1, \dots, a_m))=0$.

Again, taking arbitrary $b_1,\cdots, b_m\in M_n(K)$ and writing them as \[b_j=\displaystyle \sum_{i\in \mathbb{Z}_n}a_i^{(j)},\]
with $\deg(a_i^{(j)})=i$, for $j\in \{1, \dots, m\}$ and $i\in \mathbb{Z}_n$, the multilinearity of $f$ implies  that \[tr(f(b_1,\dots, b_m)) = tr(\sum_{i_1+\cdots +i_m\neq \overline{0}} f(a^{(1)}_{i_1}, \dots, a^{(m)}_{i_m}))=0.\] This means $Im(f)\subseteq sl_n(K)$.
\end{enumerate}
\end{proof}

From now on, we assume $K$ to be a field of characteristic $0$ or $p$ such that $p$ does not divide $n$.

Putting together the above lemma and Theorem \ref{identities}, we obtain

\begin{corollary}
    The multilinear polynomial $f(x_1,\dots, x_m)$ satisfies \emph{(S0)} if and only if $f$ satisfies both statements \emph{(S1)} and \emph{(S2)}. In particular, $f$ is a polynomial identity for $M_n(K)$ if and only if $f$ satisfies both \emph{(S1)} and \emph{(S2)}. Also, $Im(f)=K$ if and only if $f$ satisfies (S1) and do not satisfy \emph{(S2)}.
\end{corollary}

\begin{proof}
    The proof follows straightforward from Lemma \ref{S1S2} and Theorem \ref{identities} and from the fact that $sl_n(K)\cap K = \{0\}$.
\end{proof}

Summarizing the results up to now, we have for a multilinear polynomial $f$:
\begin{itemize}
    \item $f$ satisfies (S1) and (S2) if and only if $f$ is a polynomial identity
    \item $f$ satisfies (S1) and do not satisfy (S2), if and only if the image of $f$ is $K$.
    \item $f$ satisfies (S2) and do not satisfy (S1), if and only if $f$ is not an identity and the image of $f$ lies in $sl_n(K)$.
\end{itemize}

Our hope is that the last case is equivalent to $Im(f)= sl_n(K)$ and that if $f$ does not satisfy neither (S1) nor (S2) than $Im(f)=M_n(K)$.

For now, we are not able to prove this, so we weaken this to consider the linear span of the image of $f$.

\begin{proposition}\label{notS1}
    The polynomial $f$ does not satisfy \emph{(S1)} if and only if the linear span of $Im(f)$ contains $sl_n(K)$.
\end{proposition}

\begin{proof}
    The ``if part" is trivial. 
    
    Let us assume (S1) is false. Then, there exists $a_1, \dots, a_m\in M_n(K)$ homogeneous such that $\sum_{i=1}^m \deg(a_i)\neq 0$ and $f(a_1, \dots, a_m)\neq 0$. In particular, the matrix $f(a_1, \dots, a_m)$ is nonzero and  nondiagonal. Writing each $a_i$ as a linear combination of matrix units and considering that $f$ is a multilinear polynomial, opening brackets, gives us that some matrix unit $E_{i,j}$ with $i\neq j$ lies in the image of $f$. Since all the $E_{i,j}$, with $i\neq j$ are conjugated to each other, we obtain that any zero diagonal matrix lies in the linear span of the image of $f$. Again, since each trace zero matrix is equivalent to a matrix with zero diagonal by Theorem 2  of \cite{Fillmore}, we obtain that the linear span of $Im(f)$ contains $sl_n(K)$.
\end{proof}

\begin{corollary}
The multilinear polynomial $f(x_1, \dots, x_m)$ do not satisfy \emph{(S1)} and satisfies \emph{(S2)} if and only if the linear span of $Im(f)$ is $sl_n(K)$.
\end{corollary}

\begin{proof}
    It is a direct consequence of Proposition \ref{notS1} and Lemma \ref{S1S2} (2).
\end{proof}

\begin{theorem}
    The polynomial $f$ does not satisfy \emph{(S1)} and does not satisfy \emph{(S2)} if and only if the linear span of $Im(f)$ equals $M_n(K)$.
\end{theorem}    

    \begin{proof}
        The ``if part" is trivial.
        
        Assuming (S2) is false, there exist $a_1, \dots, a_m\in M_n(K)$ homogeneous with $\sum_{i=1}^m\deg(a_i)=0$ and $tr(f(a_1, \dots, a_m))\neq 0$. This means that $Im(f)$ contains a nonscalar diagonal matrix. In particular, since $Im(f)$ is closed under scalar multiplication, we obtain $Im(f)$ contains diagonal matrices of arbitrary traces. 
        
        Let now $a\in M_n(K)$ and let $b\in Im(p)$ be a diagonal matrix such that $tr(a)=tr(b)$. Then $tr(a-b)=0$, and assuming $f$ does not satisfy (S1), Proposition \ref{notS1} asserts that $a-b$ lies in the linear span of image of $f$. Since $a=b+(a-b)$, we obtain that $a$ lies in the linear span of $Im(f)$.
    \end{proof}

\begin{remark}\label{considerations}
Summarizing again the results up to now, we have for a multilinear polynomial $f$:
\begin{enumerate}
    \item $f$ satisfies \emph{(S1)} and \emph{(S2)} if and only if $f$ is a polynomial identity
    \item $f$ satisfies \emph{(S1)} and do not satisfy \emph{(S2)}, if and only if the image of $f$ is $K$.
    \item $f$ satisfies \emph{(S2)} and do not satisfy \emph{(S1)}, if and only if the linear span of the image of $f$ is $sl_n(K)$.
    \item $f$ does not satisfy neither \emph{(S1)} nor \emph{(S2)}, if and only if the linear span of the image of $f$ is $M_n(K)$.
\end{enumerate}
\end{remark} 

Let us now discuss the above situation under the hypothesis that the Lvov-Kaplansky conjecture is true. 

Of course the Lvov-Kaplansky conjecture is true if and only if the linear span of $Im(f)$ equals $Im(f)$ for each multilinear polynomial $f$. In particular, by the above remark, if the Lvov-Kaplansky conjecture is true, then we may replace the linear span of $Im(f)$ by $Im(f)$ in the last two cases. And of course, if the last two cases are still true when one replaces the linear span of $Im(f)$ by $Im(f)$, then the Lvov-Kaplansky conjecture is also true. We have just proved the following theorem.

\begin{theorem}\label{equivalence}
    The Lvov-Kaplansky is true if and only if the following assertions hold:
    \begin{enumerate}
         \item If $f$ does not satisfy \emph{(S1)} and satisfies \emph{(S2)}, then $Im(f)=sl_n(K)$.
        \item If $f$ does not satisfy neither \emph{(S1)} nor \emph{(S2)}, then $Im(f)=M_n(K)$. 
    \end{enumerate}
\end{theorem}

Now assume  $f=f(x_1, \dots, x_m)$ multilinear and $m\leq 2n-1$. Then by Remark \ref{considerations} $f$ is not an identity nor a central polynomial. In particular, $f$ does not satisfy $(S1)$ and we obtain the following

\begin{theorem}
    The Mesyan conjecture is true if the following assertion holds:
    \begin{enumerate}
        \item If $f$ does not satisfy $(S1)$ then $Im(f) \supseteq sl_n(K)$.
    \end{enumerate}
\end{theorem}

The above gives rise to the following conjecture, which is stronger than Mesyan's and weaker than Lvov-Kaplansky's.

\begin{conjecture}
    If a multilinear polynomial $f(x_1, \dots, x_m)$ does not satisfy $(S1)$ then $Im(f)\supseteq sl_n(K)$.  
\end{conjecture}

\section{Applications}

Now we give some applications of the results presented above. Our examples will be based on the results of \cite{CentronedeMello}. These are interesting examples showing that the knowledge of images of graded polynomials may be useful to understand images of ordinary polynomials.

First we will give an alternative proof for the following theorem of \cite{K-BMRJPAA}.

\begin{theorem}[Theorem 1 of \cite{K-BMRJPAA}]
   Let $f(x_1, \dots, x_m)$ be any multilinear polynomial evaluated on $n\times n$ matrices over an infinite field. Assume that $f$ is neither central nor PI. Then $Im(f)$ contains a matrix of the form  $\sum_{i=1}^nc_i E_{i,i+1}$, where $c_1\cdots c_n\neq 0$. When $char(K)$ is $0$ or prime to $n$, $Im(f)$ contains a matrix with eigenvalues $\{c,c\varepsilon, \dots, c\varepsilon^{n-1}\}$ for some $0\neq c \in K$.
\end{theorem}

\begin{proof}
    From our hypothesis we have $Im(f)\not\subseteq K$. Then, by part (1) of Lemma \ref{S1S2}, we obtain that $f$ does not satisfy (S1). In particular, there exist $a_1, \dots, a_m$ homogeneous elements with $g=\sum_{i=1}^m \deg(a_i) \neq 0$ such that $f(a_1, \dots, a_m) \neq 0$. In particular, $E_{i,i+g}\in Im(f)$, for some $g\neq 0$ in $\mathbb{Z}_n$. Since all $E_{i,j}$ with $i\neq j$ are conjugated with each other, for each $h\neq 0$, we obtain that $E_{i,i+g}\in Im(f)$ and this can be realized as an evaluation of $f$ by matrix untis $b_1, \dots, b_m$. But all matrix units are homogeneous in the Vasilovsky grading. So if we consider the algebra of $\mathbb{Z}_n$-graded polynomials $K\langle Y|\mathbb{Z}_n \rangle$ and take $y_i\in Y$ such that $\deg(y_i)=\deg(b_i)$ for $i=1, \dots, m$, then $f(y_1, \dots, y_m)$ is a nonzero $\mathbb{Z}_n$-graded polynomial of degree $h\in \mathbb{Z}_n\setminus \{0\}$. By \cite[Lemma 14]{CentronedeMello}, there exists a nonsingular matrix in the image of the graded  polynomial $f(y_1, \dots, y_m)$. This is a matrix of the form $\sum_{i=1}^nc_iE_{i,i+h}$ with $c_1, \dots, c_n\neq 0$. The above holds for any $h\neq 0$ in $\mathbb{Z}_n$. In particular, for $h=1$, we obtain matrix of the form $\sum_{i=1}^n c_iE_{i,i+1}$ with $c_1, \dots, c_n\neq 0$, and the proof is complete.
\end{proof}

As another application, we present a proof of the following theorem, which is the first part of Theorem 1 of \cite{MalevJAA}:

\begin{theorem}
   If $f$ is a multilinear polynomial evaluated on the matrix ring $M_2(K)$ (where $K$ is an arbitrary field of characteristic different from 2), then $Im(f)$ is either $\{0\}$, or $K$ (the set of scalar matrices), or $Im(f) \supseteq sl_2(K)$. 
\end{theorem}

\begin{proof}
    Assume $f$ is not a central polynomial nor an identity. Then $Im(f)\not \subseteq K$. By Lemma \ref{S1S2}, $f$ does not satisfy $(S1)$. This means that there exist homogeneous elements $a_1, \dots, a_m\in M_2(K)$ with $\sum_{i=1}^m \deg(a_i) \neq 0$ (i.e., $\sum_{i=1}^m \deg(a_i) =1$ in $\mathbb{Z}_2$)  and $f(a_1, \dots, a_m) \neq 0$. As in the previous example, set $\deg(y_i) = \deg(a_i)$. Then $f(y_1, \dots, y_m)$ is a graded polynomial of degree $1$ in $\mathbb{Z}_2$. Now we use Lemma 14 of \cite{CentronedeMello}, which states that since $f(y_1, \dots, y_m)$ is a nonzero multilinear graded polynomial of nonzero degree, then the image of $f$ contains a nonzero singular matrix of degree $1$ in $\mathbb{Z}_n$. As a consequence, the image of $f$ contains the set $(M_2(K))_1$, which is the set of all hollow matrices (matrices with zero in the main diagonal). But any trace zero matrix is equivalent to a hollow matrix.
    As a consequence, we obtain that the image of $f$ contains all trace zero matrices.
\end{proof}

As a last application of the results of the previous section, we give a new prove of \cite[Lemma 9]{K-BMR}. This lemma is a key step in the proof  the Lvov-Kaplansky conjecture for $2\times 2$ matrices over a quadratically closed field. This will be an easy consequence of the following Lemma:

\begin{lemma}
    Let $f(x_1,\dots, x_m) \in K\langle X \rangle$ be a multilinear polynomial. Then the image of $f$ evaluated on $M_2(K)$ is $sl_2(K)$ if  and only if $f$ satisfies \emph{(S2)} and $f$ do not satisfy \emph{(S1)}.
\end{lemma}

\begin{proof}
    The "only if" part is trivial.
    
    From the proof of the above lemma, if $f$ does not satisfy $(S1)$, then $Im(f) \supseteq sl_2(K)$. Also, by Lemma \ref{S1S2}, if $f$ satisfy $(S2)$, then $Im(f)\subseteq sl_2(K)$. 
\end{proof}

\begin{lemma}[Lemma 9 of \cite{K-BMR}]
    If $f$ is a multilinear polynomial evaluated on the matrix ring $M_2(K)$, then $Im(f)$ is either $\{0\}$, $K$, $sl_2(K)$, $M_2(K)$, or $M_2(K)\setminus \tilde K$, where $\tilde K$ is the set of all nondiagonalizible and non-nilpotent matrices.
\end{lemma}

\begin{proof}
    Assume $f$ is neither central nor PI. Then $f$ does not satisfy (S1). Now we have two cases to consider. If $f$ satisfies (S2) then $Im(f)$ is $sl_2(K)$, by the above lemma. If $f$ does not satisty (S2), then $Im(f)\supseteq sl_2(K)$ is a proper inclusion. In particular, there exists a diagonal matrix with nonzero trace in $Im(f)$. As a consequence, $Im(f)$ contains all diagonalizable matrices. So $Im(f)$ contains $M_n(K)\setminus\tilde K$.  Now we have two cases to consider. If $Im(f)$ contains some element of $\tilde K$ then it contains the whole $\tilde K$, since it is an irreducible invariant cone, and in this case we have $Im(f) = M_2(K)$. Otherwise, we have $Im(f)= M_n(K)\setminus \tilde K$.
\end{proof}

\section{Funding}

This work was supported by São Paulo Research Foundation (FAPESP), grant 2018/23690-6.

\bibliographystyle{abbrv}
\bibliography{ref}

\end{document}